\documentclass[12pt,leqno]{article}
\usepackage{amsfonts,amsthm,amsmath}
\newtheorem{thm}{Theorem}
\newtheorem{prop}[thm]{Proposition}
\newtheorem{lem}[thm]{Lemma}
\newtheorem{cor}[thm]{Corollary}
\theoremstyle{remark}
\newtheorem{rem}[thm]{Remark}

\newcommand{\ZZ}{\mathbb{Z}}
\newcommand{\RR}{\mathbb{R}}

\newcommand{\0}{\mathbf{0}}
\newcommand{\cF}{\mathcal{F}}

\DeclareMathOperator{\sm}{smin}

\begin{document}

\title{An Optimal Odd Unimodular Lattice in Dimension 
$72$\footnote{This work was supported by JST PRESTO program.}}

\author{
Masaaki Harada\thanks{
Department of Mathematical Sciences,
Yamagata University,
Yamagata 990--8560, Japan, and
PRESTO, Japan Science and Technology Agency (JST), Kawaguchi,
Saitama 332--0012, Japan. 
email: mharada@sci.kj.yamagata-u.ac.jp}
and 
Tsuyoshi Miezaki\thanks{Department of Mathematics, 
Oita National College of Technology, 
1666 Oaza-Maki, Oita, 870--0152, Japan. email: miezaki@oita-ct.ac.jp
}
}

\maketitle

\begin{abstract}
It is shown that 
if there is an extremal even unimodular lattice in dimension $72$,
then there is an optimal odd unimodular lattice in that dimension.
Hence, the first example of an optimal odd unimodular lattice in 
dimension $72$ is constructed from the extremal even unimodular lattice
which has been recently found by G. Nebe.
\end{abstract}

{\small

\quad
}

\section{Introduction}
A (Euclidean) lattice $L \subset \RR^n$ 
in dimension $n$
is {\em unimodular} if
$L = L^{*}$, where
the dual lattice $L^{*}$ of $L$ is defined as
$\{ x \in {\RR}^n \mid (x,y) \in \ZZ \text{ for all }
y \in L\}$ under the standard inner product $(x,y)$.
A unimodular lattice is called {\em even} 
if the norm $(x,x)$ of every vector $x$ is even.
A unimodular lattice which is not even is called
{\em odd}.
An even unimodular lattice in dimension $n$ exists if and 
only if $n \equiv 0 \pmod 8$, while 
an odd  unimodular lattice exists for every dimension.
Two lattices $L$ and $L'$ are {\em neighbors} if
both lattices contain a sublattice of index $2$
in common.

Rains and Sloane~\cite{RS-bound} showed that
the minimum norm $\min(L)$ of a unimodular
lattice $L$ in dimension $n$
is bounded by
$\min(L) \le 2 \lfloor n/24 \rfloor+2$
unless $n=23$ when $\min(L) \le 3$.
We say that a unimodular lattice meeting the upper
bound is {\em extremal}.
Gaulter~\cite{Gaulter} showed that
any unimodular lattice in dimension $24k$
meeting the upper
bound has to be even, which was conjectured by
Rains and Sloane.
Hence,  an odd unimodular lattice $L$
in dimension $24k$
satisfies $\min(L) \le 2k+1$.
We say that an odd unimodular lattice $L$ in dimension $24k$
with $\min(L)=2k+1$ is {\em optimal}.

Shadows of odd unimodular lattices appeared 
in~\cite{CS90-AMS} and \cite{CS-odd},
and shadows play an important role in the study of odd unimodular lattices.
For example, shadows are the main tool in \cite{RS-bound}.
Let $L$ be an odd unimodular lattice and let $L_0$ be the 
subset of vectors of even norm.
Then $L_0$ is a sublattice of $L$ of index 2. The {\em shadow} of $L$ is 
defined as
$S(L) = L_0^ * \setminus L$.
We define the {\em shadow minimum} of $L$
as $\sm(L) = \min \{ (x,x) \mid x\in S(L) \} $.

The aim of this note is to show the following:

\begin{thm}\label{thm:main}
If there is an extremal even unimodular lattice $\Lambda$ in dimension 
$72$, then there is an optimal odd unimodular lattice $L$ 
in dimension $72$ with $\sm(L) = 2$, which is a
neighbor of $\Lambda$.
\end{thm}

Recently Nebe~\cite{Nebe72} has found an
extremal even unimodular lattice in dimension $72$.
It was a long-standing question to determine the
existence of such a lattice.
As a consequence of Theorem~\ref{thm:main}, 
we have the following:

\begin{cor}
There is an optimal odd unimodular lattice $L$
in dimension $72$  with $\sm(L)=2$.
\end{cor}

\section{An optimal odd unimodular lattice in dimension 72}

The theta series $\theta_{L}(q)$ of a lattice $L$ is 
the formal power series
$\theta_{L}(q) = \sum_{x \in L} q^{(x, x)}$.
Conway and Sloane~\cite{CS90-AMS, CS-odd} showed that
when the theta series of an odd unimodular lattice $L$
in dimension $n$ is written as 
\begin{equation}\label{eq:CS1}
\theta_L(q)=
 \sum_{j =0}^{\lfloor n/8\rfloor} a_j\theta_3(q)^{n-8j}\Delta_8(q)^j,
\end{equation}
the theta series of the shadow $S(L)$ is written as
\begin{equation}\label{eq:CS2}
\theta_S(q)= \sum_{j=0}^{\lfloor n/8\rfloor}
\frac{(-1)^j}{16^j} a_j\theta_2(q)^{n-8j}\theta_4(q^2)^{8j}
= \sum_i B_i q^i \text{ (say),}
\end{equation}
where 
$\Delta_8(q) = q \prod_{m=1}^{\infty} (1 - q^{2m-1})^8(1-q^{4m})^8$
and $\theta_2(q), \theta_3(q)$ and $\theta_4(q)$ are the Jacobi 
theta series~\cite{SPLAG}.
As the additional conditions, it follows 
from~\cite{CS90-AMS} and~\cite{CS-odd} that
\begin{equation}
\label{eq:C}
\begin{cases}
\text{$B_r =0$ unless $r \not\equiv n/4 \pmod 2$},\\
\text{there is at most one nonzero $B_r$ for $r < (\min(L)+2)/2$},\\
\text{$B_r=0$ for $r < \min(L)/4$,}\\
\text{$B_r \le 2$ for $r < \min(L)/2$.}
\end{cases}
\end{equation}

\begin{lem}\label{lem:T}
Let $L$ be an optimal odd unimodular lattice in dimension $72$
with $\sm(L)=2$.
Then the theta series of $L$ and $S(L)$ are uniquely
determined as 
\begin{align}
\label{eq:T1}
\theta_{L}(q) =&
1 + 27918336 q^7 + 3165770864 q^8 + \cdots, \\
\label{eq:T2}
\theta_S(q) =&
2 q^2 + 127800 q^6 + \cdots,
\end{align}
respectively.
\end{lem}
\begin{proof}
In (\ref{eq:CS1}) and (\ref{eq:CS2}),
it follows from $\min(L)=7$ that
\begin{multline*}
a_0=1,
a_1=-144,
a_2=7056,
a_3=-136704,
\\
a_4=928656,
a_5=-1518336,
a_6=136704.
\end{multline*}
Since $S(L)$ does not have $\0$,
$a_9=0$.
Hence, we have the following possible theta series:
\begin{align}
\label{eq:T3}
\theta_{L}(q) =&
1 
+ (28901376 + a_7)q^7 
+ (3108623472 + a_8 - 24a_7 )q^8  + \cdots,\\
\label{eq:T4}
\theta_S(q )=&
\frac{a_8}{2^{24}} q^2
+\Big(\frac{-15 a_8}{2^{21}} - \frac{a_7}{2^{12}}\Big) q^4
+ \Big(136704 +  \frac{1767 a_8}{2^{22}} + \frac{3a_7}{2^7}\Big)q^6
+ \cdots,
\end{align}
If $x \in S(L)$ with $(x,x)=2$ then $-x \in S(L)$.
It follows from (\ref{eq:C}) that $B_2=2$ and $B_4=0$.
Hence, we have that
\[
a_7= -15 \cdot 2^{16}, a_8=2^{25}.
\]
Therefore, the theta series of $L$ and $S(L)$ are uniquely
determined.
\end{proof}


Now we start on the proof of Theorem~\ref{thm:main}.
Let $\Lambda$ be an extremal even unimodular lattice in dimension
$72$. Since $\Lambda$ has minimum norm $8$,
there exists a vector $x\in\Lambda$ with $(x,x)=8$.
Fix such a vector $x$.
Put
\[
\Lambda_x^{+}=\{v\in\Lambda\mid(x,v)\equiv0\pmod2\}.
\]
If $(x,y)$ is even for all vectors $y \in \Lambda$
then $\frac{1}{2}x \in \Lambda^*=\Lambda$ and
$(\frac{1}{2}x,\frac{1}{2}x)=2<\min(\Lambda)$, which is a contradiction.
Thus, $\Lambda_x^{+}$ is a sublattice of $\Lambda$ of index $2$,
and there exists a vector $y\in\Lambda$ such
that $(x,y)$ is odd. 
Fix such a vector $y$.
Define the lattice 
\[
\Gamma_{x,y}=
\Lambda_x^+ \cup \Big(\frac{1}{2}x+y\Big)+\Lambda_x^+. 
\]
It is easy to see that  $\Gamma_{x,y}$ is an odd unimodular
lattice, which is a neighbor of $\Lambda$.

We show that $\Gamma_{x,y}$ has minimum norm $7$.
Since $\min (\Lambda_x^+) \geq 8$,
it suffices to show that
$(u,u) \geq 7$
for all vectors $u \in (\frac{1}{2}x+y)+\Lambda_x^{+}$.
Let $u=\frac{1}{2}x+y+\alpha$ ($\alpha \in \Lambda_x^{+}$).
Then we have
\begin{equation}\label{eq:1}
\Big(u,\frac{1}{2}x\Big)=\Big(\frac{1}{2}x,\frac{1}{2}x\Big)
+\Big(y,\frac{1}{2}x\Big)+\Big(\alpha,\frac{1}{2}x\Big)
\in \frac{1}{2}+\ZZ.
\end{equation}
Here, we may assume without loss of generality that
$(u, \frac{1}{2}x) \leq -\frac{1}{2}$.
Then 
\[
\Big(u+\frac{1}{2}x, u+\frac{1}{2}x\Big)
=(u,u)+2+2\Big(u,\frac{1}{2}x\Big)
\leq(u,u)+1. 
\]
If $u+\frac{1}{2}x$ is the zero vector $\0$ then
$(u,\frac{1}{2}x)=-(\frac{1}{2}x,\frac{1}{2}x)=-2$,
which contradicts (\ref{eq:1}).
Hence, $u+\frac{1}{2}x$ is a nonzero vector in $\Lambda$.
Then we obtain $8 \leq (u,u)+1$.
Therefore, $\Gamma_{x,y}$ is an odd unimodular lattice
with minimum norm $7$, which is a neighbor of $\Lambda$.

It follows that $(\Gamma_{x,y})_0=\Lambda_x^{+}$.
For any vector $\alpha \in \Lambda_x^{+}$,
$(\frac{1}{2}x,\alpha)=\frac{1}{2}(x,\alpha) \in \ZZ$.
Hence, $\frac{1}{2}x$ is a vector of norm $2$ in $S(\Gamma_{x,y})$.
Therefore, we have Theorem~\ref{thm:main}.

\begin{rem}
A similar argument can be found in~\cite{Venkov}
for dimension $48$.
\end{rem}

By Lemma~\ref{lem:T},  the theta series of 
$\Gamma_{x,y}$ and $S(\Gamma_{x,y})$ are uniquely
determined as 
(\ref{eq:T1}) and (\ref{eq:T2}), respectively.

\begin{rem}\label{rem:Z8}
The extremal even unimodular lattice in~\cite{Nebe72},
which we denote by $N_{72}$,  
contains a sublattice 
$\{(x,\0,\0),(\0,y,\0),(\0,\0,z) \mid x,y,z \in L_{24}\}$,
where $L_{24}$ is isomorphic to $\sqrt{2}\Lambda_{24}$
and $\Lambda_{24}$ is the Leech lattice.
Since $\Lambda_{24}$ contains many $4$-frames,
$N_{72}$ contains many $8$-frames
(see e.g.\ \cite{BDHO,HMV} for undefined terms in this remark).
Take one of the vectors of an $8$-frame $\cF$ as $x$ in the construction of
$\Gamma_{x,y}$.
It follows that $\Gamma_{x,y} \supset \Lambda_x^+ \supset \cF$.
Therefore, there is a 
self-dual $\ZZ_8$-code $C_{72}$ of 
length $72$ and minimum Euclidean weight $56$
such that $\Gamma_{x,y}$ is isomorphic to 
the lattice obtained from $C_{72}$ by Construction A.
A generator matrix of $C_{72}$ can be obtained electronically from

\vspace{15pt}
http://sci.kj.yamagata-u.ac.jp/\~{}mharada/Paper/z8-72-I.txt
\end{rem}



\begin{thebibliography}{99}
\bibitem{BDHO} E.~Bannai, S.T.~Dougherty, M.~Harada and M.~Oura,
{Type~II codes, even unimodular lattices and invariant rings,}
{\sl IEEE\ Trans.\ Inform.\ Theory}
{\bf 45} (1999), 257--269.


\bibitem{CS90-AMS}J.H.~Conway and N.J.A.~Sloane, 
A new upper bound for the minimum of an integral lattice of 
determinant $1$,
{\sl Bull.\ Amer.\ Math.\ Soc.\ (N.S.)}
{\bf 23}  (1990),  383--387.

\bibitem{CS-odd}J.H.~Conway and N.J.A.~Sloane, 
{A note on optimal unimodular lattices},
{\sl J.\ Number Theory}
{\bf 72} (1998), 357--362.

\bibitem{SPLAG} J.H.~Conway and N.J.A.~Sloane,
{\sl Sphere Packing, Lattices and Groups (3rd ed.)},
Springer-Verlag, New York, 1999.

\bibitem{Gaulter}M. Gaulter, 
{Minima of odd unimodular lattices in dimension $24m$},
{\sl J. Number Theory}
{\bf 91} (2001), 81--91. 

\bibitem{Venkov} M. Harada, M. Kitazume, A. Munemasa and B. Venkov, 
{On some self-dual codes and unimodular lattices in dimension 48},
{\sl European J. Combin.}
{\bf 26} (2005), 543--557. 


\bibitem{HMV} M. Harada, A. Munemasa and B. Venkov,
Classification of ternary extremal self-dual codes of length 28,
{\sl Mathematics of Comput.}
{\bf 78} (2009), 1787--1796.

\bibitem{Nebe72}G. Nebe,
An even unimodular $72$-dimensional lattice of minimum 8,
{\sl J. Reine Angew.\ Math.},
(to appear), arXiv: 1008.2862.

\bibitem{RS-bound}E. Rains and N.J.A. Sloane, 
{The shadow theory of modular and unimodular lattices},
{\sl J. Number Theory}
{\bf 73} (1998), 359--389.

\end{thebibliography}
\end{document}